\documentclass[12pt]{amsart}

\usepackage{amssymb,amsmath,amsfonts,array}
\usepackage[numbers]{natbib}
\usepackage{hyperref}
\usepackage[perc]{overpic}
\usepackage[all]{xy}
\usepackage{todonotes}
\usepackage{comment}
\usepackage{soul}
\usepackage{cleveref}

\newtheorem{theorem}{Theorem}[section]
\newtheorem{corollary}[theorem]{Corollary}
\newtheorem{proposition}[theorem]{Proposition}
\newtheorem{lemma}[theorem]{Lemma}

\theoremstyle{definition}
\newtheorem{definition}[theorem]{Definition}

\theoremstyle{remark}
\newtheorem{remark}[theorem]{Remark}
\newtheorem{example}[theorem]{Example}

\DeclareMathOperator{\R}{\mathbb{R}}

\renewcommand{\O}{\mathcal{O}}

\newcommand{\Z}{\mathbb{Z}}

\newcommand{\mg}{M_\gamma}
\newcommand{\Es}{E^{\mathrm{s}}}
\newcommand{\Eu}{E^{\mathrm{u}}}
\newcommand{\Ec}{E^{\mathrm{c}}}
\newcommand{\Ecs}{E^{\mathrm{cs}}}
\newcommand{\Ecu}{E^{\mathrm{cu}}}
\newcommand{\Fcs}{\mathcal{F}^{\mathrm{s}}}
\newcommand{\Fcu}{\mathcal{F}^{\mathrm{u}}}
\newcommand{\Fs}{\mathcal{F}^{\mathrm{s}}_\mathcal{O}}
\newcommand{\Fu}{\mathcal{F}^{\mathrm{u}}_\mathcal{O}}

\title{Hyperbolicity of complements of orbits in Anosov flows}
\author{Sergio Fenley}
\author{Tali Pinsky}
\author{Mario Shannon}
\date{\today}

\address{Department of Mathematics, Florida State University, United States}
\email{fenley@math.fsu.edu}

 \address{School of Mathematics\\
   Monash University, Australia\\
   and Department of Mathematics, The Technion IIT, Israel}
 \email{tali.pinsky@monash.edu}

\address{Facultad de Ingenier\'ia, Universidad de la Rep\'ublica, Uruguay}
\email{mjs8822@psu.edu}

\begin{document}

\begin{abstract}
    We show that if an Anosov flow on a 3-dimensional manifold has orientable stable and unstable foliations, then the complement of any filling periodic orbit is a hyperbolic manifold. This generalizes  the known case of the complement of a closed, filling geodesic orbit in the unit tangent bundle of a hyperbolic surface. 
    Furthermore, we show that the orientability condition on invariant foliations is necessary, by constructing a counterexample in the absence of this property. 
    In  the case of the suspension flows we obtain that the complement of every collection of periodic orbits is hyperbolic, while for the geodesic flow (regardless of orientability of the invariant foliations) the complement of every filling and anannular collection of periodic orbits is hyperbolic. 
\end{abstract}

\maketitle

%-------------------------------------------------------
\section{Introduction}
An Anosov flow on a 3-manifold is essentially one with contracting
and expanding directions.
There has been much interest in these flows, both for their deep connection with the topology of the ambient manifold \cite{Anosov1967Geodesic}, and also due to their dynamical significance in the family of chaotic flows on 3-manifolds \cite{BookDynamicsBeyondUniformHyperbolicity2005}.

Within the family of Anosov flows, the most studied flows are the algebraic ones: geodesic flows on the unit tangent bundle of hyperbolic surfaces, and suspensions of automorphisms  of the two dimensional torus induced by hyperbolic matrices. Algebraic flows are well understood in terms of their statistical properties  and the topologies of their sets of closed orbits. One of the main results in this direction is that if a closed geodesic  on a closed hyperbolic surface is filling, i.e. intersects any essential simple closed curve, then a corresponding closed orbit of the geodesic flow has a hyperbolic complement \cite{FoulonHasselblatt2013Contact}. 
In an analogous way, we define in Section \ref{sec:hyperbolicity} an orbit in a 3-manifold to be \emph{filling} if it intersects any essential embedded torus.
The main result of this paper is the following theorem.

\begin{theorem}\label{thm: orbits are hyperbolic for orientable foliations}
Let $\phi$ be an Anosov flow on a 3 dimensional manifold $M$ with orientable stable and unstable foliations. Then, for every filling periodic orbit $\gamma$ of $\phi$ the complement $M\setminus\gamma$ is a hyperbolic 3-manifold.
\end{theorem}

Note that the orientabilty condition of the invariant foliations implies that the ambient manifold $M$ is necessarily orientable. Many constructions of Anosov flows start with a given Anosov flow and produce a different flow via surgeries on periodic orbits and/or cutting and pasting techniques \cite{BonattiIakovoglou2023, FranksWilliams1989, HandelThurston1980}. The knowledge that the orbit has a hyperbolic complement is often useful in these constructions \cite{FoulonHasselblatt2013Contact,BowdenMann2022Stability}.

The orientability of the invariant foliations is needed to make sure that every orbit of the flow is not \emph{divisible}, when seen as an element of the fundamental group of the manifold (cf. Section \ref{sec:divisibility}). A slight generalization of the previous theorem is the following, where the orientability condition of the invariant foliations of the flow is removed. 

\begin{theorem}
\label{thm: orbits are hyperbolic when not divisible}
Let $\phi$ be an Anosov flow on a closed and orientable 3-dimensional manifold $M$. Then, for every filling and non-divisible periodic orbit $\gamma$ of $\phi$ the complement $M\setminus\gamma$ is a hyperbolic 3-manifold.
\end{theorem}

We prove in Section \ref{sec:divisibility} that a periodic orbit $\beta$ (traversed once) of an Anosov flow may be divisible as an element of the fundamental group, only if it is homotopic to $\alpha^{\pm2}$, where $\alpha$ is another periodic orbit. Furthermore, the orbit $\beta$ must have orientable stable/unstable leaves, while the orbit $\alpha$ necessarily has non-orientable leaves. As a consequence we obtain that:

\begin{corollary}
\label{cor:main}
If a periodic orbit $\gamma$ of an Anosov flow is filling and its weak stable and weak unstable manifolds are non-orientable, then $M\setminus\gamma$ admits a hyperbolic metric.
\end{corollary}

%An example of the use of the hyperbolicity of the complement of a periodic orbit (and thus the hyperbolicity of a high enough Goodman–Fried Surgery) is used in Section \ref{sec:example} to construct an example of an Anosov flow with a filling orbit that is divisible (i.e. when the orbit itself is traversed once it is a power as an element in the fundamental group). 

Moving to the case of several periodic orbits, we show in Section \ref{sec:several orbits} that any collection of \emph{filling and anannular} periodic orbits of an Anosov flow has hyperbolic complement (cf. Theorem \ref{pro: filling anannular set is hyperbolic}). For the particular case of algebraic flows, we prove (cf. Proposition \ref{prop:several_orbits_algebraic_flow}) that 
\begin{itemize}
    \item any collection of periodic orbits in a suspension Anosov flow has hyperbolic complement;

    \item any anannular collection of filling periodic orbits in an Anosov geodesic flow has hyperbolic complement.
\end{itemize}
This is well known for geodesic flows, although does not explicitly appear in the literature for more than one orbit.

%This proves immediately that for algebraic flows Theorem \ref{thm: orbits are hyperbolic for orientable foliations} indeed holds without the condition that the foliations be orientable, via lifting the flow to a different one with orientable foliations, possibly lifting the periodic orbit to several orbits.

Finally in Section \ref{sec:example} we investigate what happens in the absence of non-divisibility. We show that Theorem \ref{thm: orbits are hyperbolic when not divisible} does not hold if we remove the non-divisibility hypothesis on the periodic orbit. More concretely, in Proposition \ref{prop:counter-example} we construct an example of a filling periodic orbit in an Anosov flow with non-orientable foliations, that is divisible as an element of the fundamental group, and whose complement is not hyperbolic. The construction begins with a geodesic flow on a non-orientable surface, and uses a \emph{derived from Anosov bifurcations} (DA) and \emph{Goodman-Fried surgeries} to obtain the required Anosov flow.

\subsection*{Acknowledgments}
This work began during a workshop at the Matrix Center in Australia in 2023; the authors thank the organizers and the center for hosting them. The second author thanks Juan Souto for exchanges about related ideas. The third author is partially supported by UdelaR-CSIC Grant 149/348 - Geometry and group actions.

%-------------------------------------------------------

\section{preliminaries}
\label{sec_preliminaries}

A flow $\phi=\{\phi^t\}_{t\in\R}$ generated by a non-singular vector field $X$ of class $C^1$ on a closed Riemannian 3-dimensional manifold $M$ is called {\it Anosov} if there is a splitting 
$TM=\Es \oplus \Ec \oplus \Eu$
of the tangent bundle of $M$ into a Whitney sum of three 1-dimensional bundles, where $\Ec$ is the line bundle colinear with the vector field $X$, which are invariant under the derivative action $D\phi^t:TM\to TM$, and there are positive constants $A$ and $B$ such that for any $t>0$,
\begin{align*}
    & \lVert D \phi^t(v^s) \rVert\leq Ae^{-Bt} \lVert  v^s\rVert,\ \ \ \text{for every}\;v^s \in \Es,\\
    & \lVert D \phi^t(v^u) \rVert\geq Ae^{Bt} \lVert  v^u\rVert,\ \ \ \ \text{for every}\;v^u \in \Eu.
\end{align*}
This kind of decomposition of the tangent bundle of the supporting manifold is known as a \emph{hyperbolic splitting}. We respectively call $\Es$ and $\Eu$ the \emph{stable} and \emph{unstable} bundles, and we call $\Ec$ the \emph{central} bundle.

Anosov flows are well-known for having a rich chaotic dynamical behavior; we refer to \cite{Katok-Hasselblatt} for an account on dynamical systems preserving a hyperbolic splitting. Among many interesting properties, Anosov flows have infinitely many periodic orbits (countably many). The hyperbolic splitting preserved by the flow implies that every periodic orbit is hyperbolic of saddle type. These flows can be either \emph{topologically transitive}, i.e. there exists a trajectory that is dense in the supporting manifold (e.g. the \emph{geodesic flow} on a closed \emph{hyperbolic} surface), or non-transitive (cf. \cite{FranksWilliams1989}).

One important property of Anosov flows, is that the 2-dimensional bundles $\Ecs=\Es\oplus\Ec$ and $\Ecu=\Ec\oplus\Eu$ integrate into a pair of codimension 1 foliations, respectively called \emph{weak stable} and \emph{weak unstable} foliations, denoted here by $\Fcs$ and $\Fcu$. These foliations are transverse to each other, invariant under the flow (so, in particular, saturated by flow-orbits), and their intersection determines a 1-dimensional foliation that coincides with the orbits of the flow.

The  weak stable (resp. weak unstable) foliation can be either orientable or non-orientable. It belongs to the class of \emph{taut foliations}, and they do not have compact leaves. Moreover, every leaf is homeomorphic to either a plane, a cylinder or a M\"obius band. A leaf is a plane if and only if it does not contain a closed orbit of the flow. In the case where a leaf is not a plane, it contains exactly one periodic orbit, non-contractible inside the leaf, and it is a cylinder whenever the $\Es$ bundle (resp. $\Eu$ bundle) is orientable along the periodic orbit, or M\"obius band otherwise. If the supporting manifold $M$ is orientable observe that, along a given periodic orbit, either both bundles $\Es$ and $\Eu$ are orientable, or both non-orientable. Finally, the center-stable and center-unstable foliations are minimal in the case where the flow is transitive, or non-minimal otherwise.

A periodic orbit of an Anosov flow, to any power, cannot be null-homotopic. This follows from the classical theorem of Novikov on Reeb components (cf. \cite{Novikov1965Foliations}), since otherwise the weak-stable/unstable foliations would contain a compact leaf. Therefore, when lifting to the universal cover of the supporting manifold, the flow-orbits lift to properly embedded lines, and the weak stable/unstable leaves lift to properly embedded planes in the universal cover. This implies, by the classical theorem of Palmeira on manifolds foliated by planes (cf. \cite{Palmeira1978}), that this universal cover is homeomorphic to $\R^3$.

Denote the universal cover of $M$ by $\widetilde{M}$. We will use the accent $\sim$ to denote the lift to the universal cover of objects in $M$. A very useful construction associated to an Anosov flow $\phi$ on $M$ is the following:

\begin{definition}
The \emph{orbit space} associated to $\phi$, denoted by $\O$, is the quotient space of $\tilde M$ obtained by identifying any two points if they lie in the same leaf of the lift of the foliation by flow-orbits. 
\end{definition}

In \cite{Fenley1994AnosovFlows, Barbot1995Characterization} it is proven that the universal cover of an Anosov flow endowed with the lift of the foliation by flow-orbits is homeomorphic to $\R^3$ endowed with the foliation by vertical lines $\{\text{point}\}\times\R$. In turn, this implies that the orbit space $\O$ is homeomorphic to the plane $\R^2$. Observe that, since the lifted foliations $\tilde{\Fcs}$ and $\tilde{\Fcu}$ are saturated by flow-lines, they induce a pair of transverse 1-dimensional foliations in the quotient, that we denote here by $\Fs$ and $\Fu$. The space $\O$ endowed with this pair of foliations is called a \emph{bi-foliated plane}.

If the leaf space of $\tilde{\Fcs}$ (resp. $\tilde{\Fcu}$) is homeomorphic to the real line, then the other leaf space $\tilde{\Fcu}$ (resp. $\tilde{\Fcs}$) is also homeomorphic to the real line, and the Anosov flow is said to be \emph{$\R$-covered}. In this case, the bi-foliated plane $(\O,\Fs,\Fu)$ can be of two types up to topological equivalence: Either $\R^2$ endowed with the foliations by vertical/horizontal lines, or an open band in $\R^2$ delimited by two parallel lines of slope $\neq 0,\infty$ endowed with the restriction of the foliations by vertical/horizontal lines. In the first case, the bi-foliated plane is said to have a \emph{global product structure} and the flow is necessarily orbit equivalent to a \emph{suspension} Anosov flow (cf. \cite{Barbot1995Characterization}). In the second case, the bi-foliated plane is called \emph{skewed $\R$-covered}, and the flow is orbit equivalent to a \emph{contact Anosov flow} (cf. \cite{marty2024skewedanosovflowsorbit}). There is a big family of Anosov flows entering into this second category, that includes all geodesic flows of hyperbolic surfaces. In general, an Anosov flow need not to be $\R$-covered, which means that the leaf spaces of the invariant foliations must contain \emph{non-separated leaves}.

\begin{figure}[t]
    \centering
    \begin{overpic}[width=12 cm]{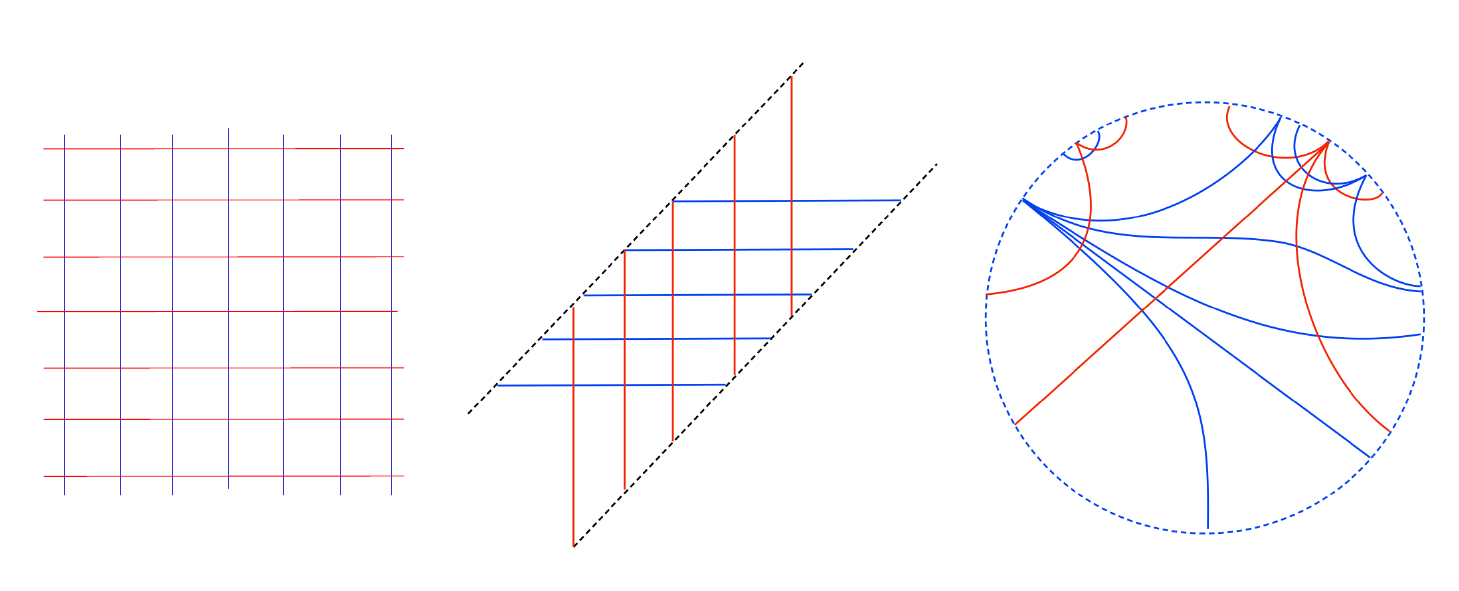}
    \end{overpic}
    \caption{The bi-foliated orbit space $(\mathcal{O},\Fs,\Fu)$. Left: global product structure; Center: skewed $\mathbb{R}$-covered; Right: general case having non-separable leaves.}
    \label{fig:action}
\end{figure}

Let $\pi_1(M,x_0)$ denote the fundamental group of $M$ based at some point $x_0$. This group acts on $\tilde{M}$ preserving the lift of the foliation by flowlines and preserving $\tilde{\Fcu}$ and $\tilde{\Fcs}$, and therefore it induces an action on $\O$ preserving the 1-dimensional foliations $\Fu$ and $\Fs$. The conjugacy class of the action of the fundamental group on the bi-foliated plane is an invariant of the Anosov flow up to orbital equivalence (cf. \cite{Barbot1995Characterization}), and it turns to be very helpful in the analysis of these dynamical systems. We will use it in Section \ref{sec:divisibility}.

An important operation on Anosov flows is the so-called \emph{Goodman-Fried surgery}. Given an Anosov flow $(\phi,M)$ on a 3-manifold, a periodic orbit $\gamma$ and an integer $m$, this operation allows to construct a new Anosov flow on a new 3-manifold, by performing a \emph{Dehn surgery} of slope $m$ along the closed curve $\gamma$ with hyperbolic complement. The surgery is adapted to the flow in such a way that the new manifold is endowed with an Anosov flow. This operation has been central in the construction of non-classical Anosov flows. For instance, Goodman obtained the first examples of Anosov flows on hyperbolic 3-manifolds, by performing this operation along a periodic orbit. We refer to the classical references \cite{Goodman1983DehnSurgery} and \cite{Fried1983Transitive}, as well as to \cite{shannon_thesis} for a unified point of view of Goodman surgery and Fried surgery.

A 3-dimensional manifold $M$ is \emph{hyperbolic} if it admits a \emph{hyperbolic metric}, i.e. a complete Riemannian metric of constant sectional curvature equal to $-1$. One of the cornerstones of the theory of 3-dimensional manifolds is the fact, that the existence of a hyperbolic metric in 3-dimensional manifold can be characterized through purely topological properties. For the case of interest of this text, the result can be stated as follows:

\begin{theorem}[Thurston Hyperbolization]
\label{thm_hyperbolization}
Let $M$ be a compact and orientable 3-manifold with non empty boundary, and  whose boundary components are tori. Then, $M$ admits a complete hyperbolic metric if and only if it is \emph{irreducible}, \emph{boundary irreducible}, \emph{atoroidal} and \emph{anannular}. 
\end{theorem}

A compact 3-manifold $M$ is:
\begin{enumerate}
    \item \emph{Irreducible} if every embedded sphere $S\hookrightarrow M$ is the boundary of a 3-dimensional ball.

    \item \emph{Boundary irreducible} if for any embedded disk  $(D,\partial D)\hookrightarrow (M,\partial M)$, $\partial D$ also bounds a disk in $\partial M$. 

    \item \emph{Atoroidal} if every embedded torus $T\hookrightarrow M$ either admits a compressing disk (an embedded disk $D\hookrightarrow M$ such that $\partial D\hookrightarrow T$), or else is \emph{boundary parallel} torus (isotopic to a boundary component of $M$).

    \item \emph{Anannular} if any embedded annulus $(A,\partial A)\hookrightarrow (M,\partial M)$ either admits a compressing disk, or a boundary compression where $\partial D$ can be decomposed as $\alpha\subset A$ and $\beta\subset \partial M$, each connecting the two boundary components.
\end{enumerate}

\begin{remark}
%If an embedded torus in a 3-manifold is non-essential, it is the boundary of an embedded solid torus. 
Irreducibility and boundary irreducibility implies that if an embedded annulus $(A,\partial A)\hookrightarrow (M,\partial M)$ is non-essential (i.e. has a compressing disk), then it is boundary parallel.    
\end{remark}

See to \cite{ThurstonThreeDymensionalManifoldsKleinian} and \cite{kapovich2001hyperbolic} for an account on hyperbolic 3-manifolds and Theorem \ref{thm_hyperbolization}. See \cite{76Hempel3MfdsBook} for essential surfaces in 3-dimensional manifolds and related concepts, and \cite{Purcell_Knot_theory} for a general account on knot-hyperbolicity. 

%-------------------------------------------------------
\section{Primeness}
\label{sec:primeness}

\begin{definition}
A knot $K$ in a three-dimensional manifold $M$ is \emph{prime} if for any sphere $S$ that $K$ intersects at exactly 2 points there is a disk embedded in $M$, so that its interior does not intersect $S$ or $K$, and so its boundary is a union of an arc in $K$ and an arc in $S$. 
\end{definition}

In other words this means we can slide the knot along the
disk so that the sphere does not intersect the isotoped
knot. 

\begin{definition}
    A knot $K$ in a three-dimensional manifold $M$ is \emph{hyperbolic} if $M\setminus K$ admits a complete hyperbolic metric.
\end{definition}

A hyperbolic knot $K$ in an irreducible manifold is necessarily also prime, 
 as any sphere in $M$ punctured twice by the knot is an annulus in the knot complement. As a hyperbolic manifold is anannular, the annulus is either compressible or boundary compressible, implying that the knot is trivial in a ball on one side of the sphere. 
 
 Therefore, we begin by proving that an orbit in an Anosov flow is always prime in the appropriate sense. This is weaker than hyperbolicity, but holds with no assumptions on the orbit, and will be a step in proving hyperbolicity in subsequent sections.

\begin{lemma}\label{lem:prime}
Let $\phi$ be an Anosov flow on  3-dimensional manifold $M$. Then any periodic orbit of $\phi$ is a prime knot.
\end{lemma}

\begin{proof}
Let $\gamma$ be a periodic orbit of $\phi$, and consider a sphere $S\subset M$ intersecting $\gamma$ at exactly two points. Consider the universal cover $\tilde M$ of $M$. As we explained at Section \ref{sec_preliminaries}, the fact $M$ carries an Anosov flow implies $\tilde M\cong\R^3$. The lift to the universal cover $\tilde M$ of any orbit $\gamma$ is an infinite line bounding a half plane, the plane being the lift of a half--stable leaf through $\gamma$. Observe also that, the fact that $\tilde M\cong\R^3$ implies that a 3-manifold carrying an Anosov flow is always irreducible (see \cite{Rosenberg1968} and \cite{Novikov1965Foliations}).

Now, start with an embedded sphere $S$ intersecting the knot $K$ in two points. This sphere bounds a unique ball $B$ 
in $M$. Put the sphere in general position with respect to the weak stable foliation, so now the intersection of 
$\gamma$ with $S$ bounds a single arc $\delta$ contained in the intersection of $S$ with a half leaf $L$ of the weak stable manifold of $\gamma$. Inside $L$, the union $Z=\delta\cup (\gamma \cap B)$ is a closed curve, that is also contained in $B$. By lifting everything to $\tilde M$, we can see that $Z$ lifts to a closed curve $\tilde Z$ which bounds a disk $\tilde D$ in $\tilde L$,
so its projection to $M$ bounds a disk $D$ contained in the weak stable leaf of $\gamma$. 
Now, without modifying the intersection $S\cap \partial D$, put $S$ in general position with respect to $D$, so the intersection of $S$ and $D$ is a union of simple closed curved, being $Z=S\cap\partial D$ one of its components. If $S\cap D$ only contains $Z$, the the disk $D$ is contained in $B$ and we are done. Otherwise, observe that each of the curves different from $Z$ bounds a disk contained in $S$. Therefore, by using that $M$ is irreducible, standard \emph{innermost disk arguments} show that each of these
can be inductively removed, ending up with a disk $D'$ contained in $B$ and whose boundary coincides with $Z$. This shows that $\gamma$ is prime.
\end{proof}

\begin{remark}\label{rem:several_arcs}
    Note that one can consider in a similar way a ball $B$ intersecting an orbit $\gamma$ along a few arcs. The arcs are jointly trivial in the sense that they can be isotoped to $\partial B$ simultaneously. This can be done as the local weak stable manifolds do not intersect.
\end{remark}

%-------------------------------------------------------
\section{Divisibility of periodic orbits}
\label{sec:divisibility}

In preparation for studying whether the complement of a filling periodic orbit is hyperbolic, we need to study the relationship between periodic orbits of an Anosov flow and the fundamental group of the supporting manifold.

Let $\gamma$ be a periodic orbit of an Anosov flow $(\phi,M)$. We always consider it as a simple closed curve in the 3-manifold $M$; that is, the image of an injective parametrization $\gamma:[0,1]\to M$. Observe that every periodic orbit $\gamma$ of the flow determines a conjugacy class in $\pi_1(M,x_0)$, obtained by joining $\gamma$ to the base point $x_0$ using arbitrary paths in $M$. Let us denote by $c(\gamma)$ this conjugacy class. By choosing some element $b\in c(\gamma)$ we can write $c(\gamma)=[b]$, where the bracket $[\cdot]$ is used to denote the conjugacy class of an element of the fundamental group.

An element $b\in\pi_1(M,x_0)$ is called \emph{non-divisible} if the equality $b=a^n$ with $a\in\pi_1(M,x_0)$ implies $n=\pm 1$. We are interested in the following property of Anosov flows:

\begin{proposition}\label{prop_divisibility_per_orbits}
    Let $\gamma$ be a periodic orbit of an Anosov flow $(\phi,M)$ on an orientable 3-manifold, and let $b\in c(\gamma)$. Consider $a\in\pi_1(M,x_0)$ and $n\in\Z$ satisfying $b=a^n$. Then $|n|=1$ or $|n|=2$. Moreover, if $b=a^{\pm 2}$ holds, then the weak stable and weak unstable manifolds of $\gamma$ are orientable, and there exists a periodic orbit $\delta$ with non-orientable weak manifolds such that $\gamma$ is freely homotopic to $\delta^2$.
\end{proposition}

This proposition follows from the analysis of the fundamental group action on the orbit space associated to the Anosov flow. One particularly useful structure in the orbit space $\O$ is the following:

\begin{definition}
A \emph{lozenge} is a region $L\subset\O$ bounded by two half-leaves $l^\mathrm{s}_z$ and $l^\mathrm{u}_z$ of $\Fs$ and $\Fu$ passing through a point $z$, and two half-leaves $l^\mathrm{s}_w$ and $l^\mathrm{u}_w$ passing through a point $w$ that are disjoint from $l^\mathrm{s}_z$ and $l^\mathrm{u}_z$, such that for every $p$ in the interior of $L$ the $\mathrm{s}$-leaf through $p$ intersects both $l^\mathrm{u}_z$ and $l^\mathrm{u}_w$ and the $\mathrm{u}$-leaf through $p$ intersects both $l^\mathrm{s}_z$ and $l^\mathrm{s}_w$. 
\end{definition}

The points $z$ and $w$ are called the \emph{corners} of the lozenge. The interior of the lozenge has a \emph{product structure} with respect to the $\mathrm{s}-$ and $\mathrm{u}$-foliations. Its closure in $\O$ is homeomorphic to a compact rectangle with two opposite corners missing. See details in \cite{Fenley1998StructureBranching}.

Observe that, for a skewed $\R$-covered Anosov flow, every point in the orbit space is the corner of a lozenge. In general, this needs not to be true, as we can see in the case of an $\R$-covered Anosov flow with global product structure.

\begin{definition}
    A \emph{chain of lozenges} is a sequence of lozenges $\{L_i\}_{i\in I}$ (where $I=\Z,\ \mathbb{N}$ or a finite set $\{1, \cdots, n\}$) such that, for every $i\in I$, the lozenges $L_j$ and $L_{j+1}$ share either a corner or a half-leaf in their boundaries. 
\end{definition}

An important property of the fundamental group action on the orbit space is that, if an element $a$ of the fundamental group fixes a point $z\in\O$, then $z$ represents the lift to the universal cover of a periodic orbit $\gamma$ and it is verified that $[a]=c(\gamma)$ (cf. \cite{Barbot1995Characterization}). We will need the following result:

\begin{theorem}[Fenley, \cite{Fenley1995QuasigeodesicAndPropertiesOfFlowlines}]\label{thm: Chain of Lozenges}
    If two points $z$ and $w$ in $\O$ are fixed under the action of a non-trivial element of $\pi_1(M,x_0)$, then $z$ and $w$ are connected by a finite chain of lozenges (i.e. each of them is a corner for a lozenge at the end of the chain).
\end{theorem}

Proposition \ref{prop_divisibility_per_orbits} is a consequence of the following lemma:

\begin{lemma}\label{lem:power<=2}
     Let $a$ be an element of $\pi_1(M,x_0)$ and $n\neq 0$ an integer. If the action of $a^n$ on $\O$ has a fixed point $z$, then $z$ is also fixed by $a^2$. Moreover, if the action of $a$ on $\O$ fixes the orientation of both foliations $\Fs$ and $\Fu$, then $z$ is also fixed by $a$. 
\end{lemma}

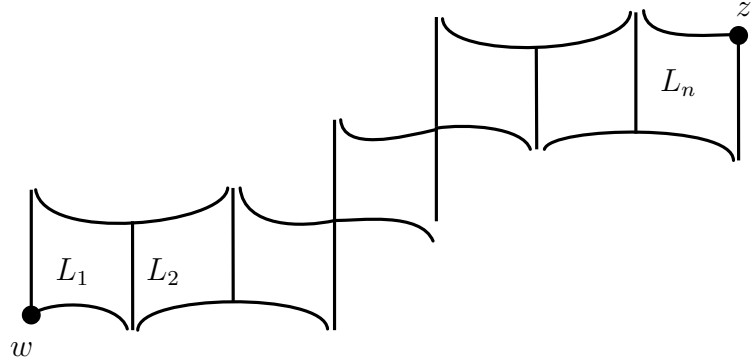
\begin{figure}[t]
    \centering
    \begin{overpic}[width=10 cm]{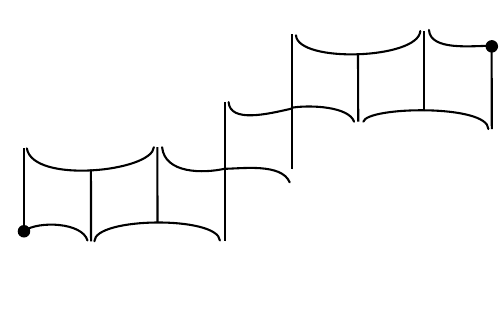}
        \put(2,15){$w$}
        \put(8,25){$L_1$}
        \put(20,25){$L_2$}
        \put(88,50){$L_n$}
        \put(98,60){$z$}
    \end{overpic}
    \caption{A chain of lozenges connecting the fixed point $w$ of $a$ and the fixed point $z$ of $a^2$.}
    \label{fig:chain}
\end{figure}

\begin{proof}
Since we assume the 3-manifold $M$ is orientable, then every $a\in\pi_1(M,x_0)$ acts preserving the orientation of the plane $\O$. To see this, recall that by \cite{Fenley1994AnosovFlows, Barbot1995Characterization} the universal cover of $M$ endowed with the lift of the foliation by flow-orbits is isomorphic to $\R^3$ endowed with the foliation by vertical lines $\{\text{point}\}\times\R$. Observe that the orbits of the flow are naturally oriented by the flow-direction, and hence their lift to the universal cover is an oriented foliation. Since $a$ acts on the universal cover preserving its orientation and also preserving the orientation of the lift of the foliation by flow-orbits, then it also preserves the orientation of every transverse plane $\R^2\times\{\text{point}\}$. Taking the quotient by vertical lines, this induces an orientation preserving homeomorphism of the orbit space. In practice, this translates into the following property: The action of $a$ on $\O$ either preserves the orientation of both foliations $\Fs$ and $\Fu$, or either reverses both.

Let $n\neq 0$ be an integer and let $z\in\O$ be a point satisfying $a^n(z)=z$. A classical theorem of Brower about plane homeomorphisms states that, if an orientation preserving homeomorphism of the plane has a periodic point, then it has a fixed point (see \cite{Franks_BrowerThm} for a proof). Applying this to our context we obtain that there exists a point $w\in\O$ such that $a(w)=w$. If $z=w$ the lemma follows directly, so assume $z\neq w$. Since $z$ and $w$ are both fixed by $a^n$, Theorem \ref{thm: Chain of Lozenges} implies that these two points are connected by a chain of lozenges $\mathcal{C}=\{L_1,\dots,L_n\}$.

Consider the action of $a^2$ on $\mathcal{C}$. Since $a^2$ fixes $w$ and preserves the orientations of $\Fs$ and $\Fu$, then $a^2$ fixes the sides of the lozenge $L_1$ containing $w$, which implies that it fixes the entire lozenge $L_1$. Since $a^2$ fixes the sides of $L_1$ that contains the corner opposite to $w$, it fixes the sides of the lozenge $L_2$ adjacent to $L_1$, and hence the entire lozenge $L_2$. (There are two cases to consider here: Either $L_1$ and $L_2$ share a corner or share an entire side.) By induction, $a^2$ fixes all lozenges $L_1\dots L_n$ connecting $w$ with $z$. Hence $a^2(z)=z$, as required.

Finally, if $a$ preserves the orientation of the stable/unstable leaves containing $w$, then $a$ fixes all the lozenges in the chain connecting $w$ with $z$, and hence $z$ is fixed by $a$.
\end{proof}

\begin{proof}[Proof of Proposition \ref{prop_divisibility_per_orbits}]
Let $\gamma$ be a periodic orbit and let $b\in\pi_1(M,x_0)$ such that $c(\gamma)=[b]$. This element $b$ is represented by a closed curve $\beta=\epsilon\cdot\gamma\cdot\epsilon^{-1}$ in the 3-manifold $M$, where $\epsilon$ is an arc joining the base point $x_0$ with a point $y_0$ in the simple closed curve $\gamma$. Choose a lift $\tilde{x}_0\in\tilde{M}$ of the base point $x_0$ and consider the action of $\pi_1(M,x_0)$ on $\tilde{M}$ obtained by lifting closed curves with starting point at $\tilde{x}_0$. Let $\tilde{\epsilon}:[0,1]\to\tilde{M}$ be the lift starting at $\tilde{\epsilon}(0)=\tilde{x}_0$ and denote $\tilde{y}_0=\tilde{\epsilon}(1)$. By construction, the point $\tilde{y}_0$ belongs to a component $\tilde{\gamma}_0$ of the lift of $\gamma$ to $\tilde{M}$. Then, the action of $b$ on the point $\tilde{y}_0$ is, by definition of the fundamental group action, the point 
$\tilde{y}_1=(\beta\cdot\epsilon)^{\sim}(1)=(\epsilon\cdot\gamma)^{\sim}(1)$, which is easily seen to belong to $\tilde{\gamma}_0$. Therefore, $b$ fixes the component $\tilde{\gamma}_0$ and hence it fixes the associated point $z=(\text{class of }\tilde{\gamma}_0$) in $\O=\tilde{M}/\text{orbit foliation}$.

Consider the equation $b=a^n$ and let $\alpha:[0,1]\to M$ be a closed curve with $\alpha(0)=\alpha(1)=x_0$ representing the element $a$. By Lemma \ref{lem:power<=2} we have that, when acting on the orbit space, either $a(z)=z$, or $a(z)\neq z$ and $a^2(z)=z$.

If $a$ fixes $z$, consider the action of $a$ on the point $\tilde{y}_0$. By definition of the fundamental group action, this is the point 
$\tilde{y}_2=(\alpha\cdot\epsilon)^{\sim}(1)$. Since $a$ fixes the component $\tilde{\gamma}_0$, there exists a segment $\tilde{\zeta}$ joining $\tilde{y}_0$ with $\tilde{y}_2$ contained in the line $\tilde{\gamma}_0$, and the concatenation $\tilde\alpha\cdot\tilde\epsilon\cdot\tilde{\zeta}^{-1}\cdot\tilde\epsilon^{-1}$ is a closed curve in the universal cover, so it is homotopically trivial. Projecting on the 3-manifold $M$, we obtain that $\alpha$ is homotopic to $\epsilon\cdot\zeta\cdot\epsilon^{-1}$. Since $\zeta$ is a curve parametrizing $\gamma$ a number of times $m\neq 0$, this implies that $\alpha$ is homotopic to $\beta^m$, and hence $b=b^{nm}$. Finally, since no non-trivial power the periodic orbit $\gamma$ can be null-homotopic, we obtain that $nm=\pm 1$, from where $n=\pm 1$. Therefore, $b=a^{\pm 1}$.

If $a(z)\neq z$ and $a^2(z)=z$, we can apply the same argument as in the previous paragraph replacing the role of $\alpha$ by $\alpha^2$, and obtain that $b=a^{\pm 2}$. In this case, Lemma \ref{lem:power<=2} guarantees the existence of another point $w\in\O$ such that $a(w)=w$, where the action of $a$ necessarily switches the orientations of the leaves of the foliations $\Fs$ and $\Fu$. It follows that there exists a periodic orbit $\delta$ satisfying $c(\delta)=[a]$, but whose invariant manifolds are non-orientable. Finally, since $\delta^2$ and $\gamma$ define the same conjugacy class in $\pi_1(M,x_0)$, these two closed curves are freely homotopic inside $M$. 
\end{proof}

\begin{remark}
    A special case of Proposition \ref{prop_divisibility_per_orbits} is the case of geodesic flows $\phi_t:T^1\Sigma\to T^1\Sigma$ defined on non-orientable hyperbolic surfaces $\Sigma$. For every periodic orbit $\gamma$ of such a flow, Proposition \ref{prop_divisibility_per_orbits} says that either $c(\gamma)=[a^{\pm1}]$ or $\gamma=[a^{\pm2}]$, where $a$ is a non-divisible element of $\pi_1(T^1\Sigma,x_0)$. Let's see that, in fact, it holds $c(\gamma)=[a^{\pm1}]$ for every periodic orbit $\gamma$. For this, assume $c(\gamma)=[a^{\pm2}]$ for some $a$ in the fundamental group. Then, by Proposition \ref{prop_divisibility_per_orbits} there exists another periodic orbit $\delta$ such that $c(\delta)=[a]$ and $\delta^2$ is freely homotopic to $\gamma$. Now, if we project these two orbits of the geodesic flow on the surface $\Sigma$, we obtain two closed geodesics $g_\delta$ and $g_\gamma$  such that $g_\gamma$ is homotopic to the geodesic $g_\delta$ traversed twice. Since in a hyperbolic surface $\Sigma$ there is at most one closed geodesic per free homotopy class, then $g_\delta^2=g_\gamma$ and we conclude that the orbit $\gamma$ is equal to the orbit $\delta$ traversed twice. Since every periodic orbit is a simple closed curve (i.e. injective parametrization $\gamma:[0,1]\to M$), we obtain a contradiction from the assumption $c(\gamma)=[a^2]$. 
\end{remark}

\section{Hyperbolicity of periodic orbits}
\label{sec:hyperbolicity}

In this section we consider the problem of determining whether a periodic orbit of an Anosov flow is a hyperbolic knot in the supporting 3-manifold.

First, note that we cannot expect the complement of a closed orbit of an Anosov flow to be always hyperbolic. For example, consider the geodesic flow $\phi_t:T^1 \Sigma\to T^1 \Sigma$ on the unitary tangent bundle of a closed hyperbolic surface $\Sigma$. If $\gamma$ is a periodic orbit corresponding to a non-filling closed geodesic, there exists a simple closed geodesic $\kappa$ in $\Sigma$ such that the projection of $\gamma$ on $\Sigma$ does not intersect $\kappa$, and then the set $T^1_\kappa \Sigma$ consisting of all unitary vectors in $\Sigma$ with base point in $\kappa$ is an essential torus embedded in $T^1\Sigma\setminus\gamma$. Hence, $T^1 \Sigma\setminus\gamma$ cannot be hyperbolic.

\begin{definition}\label{def:filling orbit}
    Let $\gamma$ be a simple closed curve in an irreducible 3-manifold $M$. We say that $\gamma$ is \emph{filling} if it has nonempty intersection with every essential torus or essential Klein bottle in $M$. 
\end{definition}

\begin{remark}
    In the case of orientable 3-manifolds, in order to prove that a simple closed curve is filling it is enough to show that it intersects non-trivially every essential torus. This is because, in an orientable 3-manifold, the boundary of a tubular neighborhood of an essential Klein bottle is an essential torus. Hence, if a simple closed curve has non-empty intersection with every essential torus, in particular it has non-empty intersection with the boundary of every small tubular neighborhood of an essential Klein bottle, and this forces an intersection with the Klein bottle. 
\end{remark}

We will prove that the periodic orbits of a transitive Anosov flow which are filling and non-divisible are hyperbolic knots. Let us show first that any transitive Anosov flow has infinitely many of these orbits.

\begin{proposition}
    Every transitive Anosov flow has infinitely many filling and non-divisible periodic orbits.
\end{proposition}

\begin{proof}
We only need to check intersection with essential tori and Klein bottles. We do the proofs for tori, and similar proof works for Klein bottles.
 Given an orbit $\gamma$, then $\gamma$ being filling is equivalent to $\gamma$ intersecting any torus of the JSJ decomposition, and any vertical torus in a Seifert fibered piece. i.e., it's projection to the orbit surface of any such piece is filling on the surface.

   First notice that for any Birkhoff torus, an orbit intersects it efficiently:  in other words no segment in $\gamma$ with endpoints
   in the torus can be homotopically pushed into the torus unless the
   whole orbit is contained in the torus. This is better seen in the
   universal cover: if $\tilde T$ is the lift of the torus and
   $\tilde \beta$ is the lift of an orbit $\beta$, then either 
   $\tilde \beta$ is contained in $\tilde T$ or $\tilde \beta$ is 
   intersects $\tilde T$ in a single point. This is because stable and
   unstable leaves in the universal cover intersect $\tilde T$ in connected sets. \cite{BarbotFenley2021Free}.
   
   Then for each torus of the JSJ decomposition choose an open set of directions pointing out, and another open set with a range of directions going in.

   For a Seifert fibered piece, it is enough to look at intersection
   with essential embedded tori in the piece. Hence look at the base surface. Fix a train track carrying all simple closed curves on the orbit space/orbifold, and for each edge of the train track choose on open set in the unit tangent bundle going through the midpoint of each edge with roughly perpendicular direction. 

   Choose a dense orbit, follow it for a finite time until it passes through all these open sets and back to close to where it started. Find a closed orbit approximated by this dense segment using the Anosov closing lemma. This orbit intersects all essential tori, and all tori that are the tangent bundles to simple closed curves on the orbit surfaces. By the  claim above these intersections cannot be removed by isotopy/homotopy, and thus the orbit is filling.

   Using more and more segments of dense orbits returning to the same first open neighborhood, passing the open neighborhood in different orders (to ensure we do not get powers of the same orbit), yields there are infinitely many such orbits.
\end{proof}

In preparation for the proof of Theorem \ref{thm: orbits are hyperbolic for orientable foliations} we first  describe how a periodic orbit of an Anosov flow looks-like when contained inside a solid torus.

\begin{lemma}\label{lem:isotopic to core}
Let $\gamma$ be a periodic orbit of an Anosov flow $\phi$ on a closed 3-manifold $M$. If $\gamma$ is contained in a solid torus and is homotopic to a core $\alpha$ of the solid torus, then $\gamma$ and $\alpha$ are isotopic.
\end{lemma}

\begin{proof}
Let $V$ be the solid torus and $T$ its boundary. They will be fixed
throughout the proof. Consider the invariant manifold $L^s(\gamma) \cap V$ containing $\gamma$. Lifting to the universal cover $\tilde M$, $L^s(\gamma)$ lifts to a plane $\tilde L$, through which $\tilde\gamma$ is an infinite line, and it has at least two infinite components of the intersection with the infinite cylinder $\tilde T$, one on either side of $\gamma$.

 Isotope $\tilde L$ equivariantly in the complement of $\tilde\gamma$ so that it is transverse to $\tilde T$. Then, using that $M$ is irreducible, isotope it so that all
intersections with $T$ are not null homotopic. 
It follows that there is an annulus $S$, the quotient of the innermost component of $L_1\setminus \tilde T$ embedded in $V$ with core $\gamma$.

This annulus defines an isotopy from $\gamma$ to a curve $\gamma'\subset T$.  
Let $D$ be a meridian disk in $V$, well defined up to isotopy.
By the assumption that $\gamma$ and therefore $\gamma'$ is homotopic to the core of $V$, the arcs of $\gamma'\setminus D$ can be removed, innermost first, by an isotopy of $\gamma'$ within $T$, so that the resulting curve $\gamma''$ has a single intersection with $D$. By isotoping $\gamma''$ further, we may assume it is monotone with respect to the fibration of $T$ by curves parallel to $\partial D$.
Observe that, in an homology basis $\{l,m\}$ of $H_1(T)$ where $m=[\partial D]$ and $l$ is the class of a closed loop intersecting $\partial D$ only once, we have that $[\gamma"]=(1,n)$ for some $n\in\Z$. Hence, $\gamma''$ is isotopic to a core $\alpha$ of the solid torus $T$ by sliding it inwards the solid torus $V$ along meridians disks. Therefore, concatenating the all isotopies 
$$\gamma\to\gamma'\to\gamma ''\to\alpha$$ 
we conclude that $\gamma$ is isotopic to $\alpha$, as required.
\end{proof}

We are finally ready to prove Theorem \ref{thm: orbits are hyperbolic for orientable foliations}. The proof is divided into four lemmas, analyzing whether the complement a filling periodic orbit contains no essential spheres, disks, tori or annuli. 

Given a simple closed curve $\gamma$ in a 3-manifold $M$, we denote by $\mg=M\setminus\mathrm{int}(W)$ the manifold obtained by removing the interior of a small tubular neighborhood of $W$ around $\gamma$. 

\begin{lemma}\label{lem:irreducible}
    Let $\gamma$ be a periodic orbit of an Anosov flow $\phi$ on a closed 3-manifold $M$. Then the 3-manifold $\mg$ is irreducible.
\end{lemma}

\begin{proof}
Let $S$ be an embedded sphere in $\mg$. We need to prove that $S$ bounds a ball on one of its sides. Observe that $S$ can be viewed as a sphere in $M$ which does not intersect $\gamma$. Since a 3-manifold carrying an Anosov flow is always irreducible then $S$ bounds a ball $B$ in $M$. If the ball $B$ does not intersect $\gamma$ then $S$ also bounds a ball in $\mg$ and the result follows. If not, since $\gamma$ does not intersect $S$ then $\gamma$ must be contained in $B$, and this implies that $\gamma$ is homotopically trivial, which is not possible for a periodic orbit of an Anosov flow (cf. Section \ref{sec_preliminaries}).
\end{proof}

\begin{lemma}\label{lem:atoroidal}
    Let $\gamma$ be a filling periodic orbit of an Anosov flow $\phi$ with orientable foliations on a 3-manifold $M$. Then the complement $\mg$ is atoroidal, i.e. every embedded two dimensional torus is either compressible or boundary parallel.
\end{lemma}

\begin{proof}
Let $T$ be an embedded torus in $\mg$. We need to prove that if $T$ is essential in $\mg$ then it is boundary parallel. $T$ can be viewed as a torus in $M$ not intersecting $\gamma$. Since we assume $\gamma$ is filling it follows that $T$ is non-essential when viewed inside $M$ and thus it has a compression disk $D$ embedded in $M$ whose boundary $\partial D$ is contained in $T$ and is nontrivial in $T$. Note that if $\gamma\cap D=\emptyset$ then the disk $D$ is embedded in $\mg$ and hence $T$ is compressible in $\mg$. Therefore, we may assume $\gamma$ has non-trivial intersection with $D$.

Consider the surface $S$ obtained by surgering the torus $T$ along the disk $D$. 
Then $S$ is homeomorphic to a sphere and, since $M$ is irreducible, then $S$ is the boundary of a ball $B$ in $M$.
We have two cases to consider, depending on the location of the disk $D$ with respect to $B$ (see, for instance, \cite{Hatcher_3mflds}).
\begin{itemize}
    \item[(a)] \emph{The disk $D$ is in the interior of the ball $B$}. In this case we have that $T$ is contained in the ball $B$, and this torus can be seen as the boundary of a ``ball with a knotted tunnel'' as in the left side of Figure \ref{fig:tunnel}, the tunnel being $N(D)$. Since $\gamma$ intersects transversely the disk $D$ then $\gamma\cap N(D)$ consists in a finite set of arcs connecting $D_{-\varepsilon}$ with $D_\varepsilon$. By means of Lemma \ref{lem:prime}, each arc $\beta$ in this intersection has a boundary compression, i.e., a disk $R$ whose interior is embedded in $B$ and whose boundary decomposes as a union $\partial R=\beta\cup\delta$, where $\delta$ is an arc contained in $\partial B$. This means that the tunnel $N(D)$ is actually unknotted inside $B$ as in the right side of Figure \ref{fig:tunnel}, and a compression disk for the tunnel determines a second compression disk for the torus $T$. This disk is separated from $\gamma$ by $T$, and thus is embedded in $\mg$. Hence, the torus $T$ is non-essential in $\mg$.

    \item[(b)] \emph{The disk $D$ is disjoint with the ball $B$}. In this case the union $V=B\cup N(D)$ is a solid torus and $\gamma$ is contained in its interior. Since the fundamental group of $V$ is isomorphic to the infinite cyclic group generated by a core $\alpha$ of $V$, there exists $n\in\Z$ such that $\gamma$ is homotopic to $\alpha^n$. Observe that $n\neq 0$, since otherwise $\gamma$ would be null-homotopic, which contradicts the fact that $\gamma$ is a periodic orbit of an Anosov flow (cf. Section \ref{sec_preliminaries}). 

    As $\gamma$ is a periodic orbit, the action of $\gamma=\alpha^n\in\pi_1(M)$ on $\O$ has a fixed point $\bar z$.
Therefore, by Lemma~\ref{lem:power<=2}, $|n|$ is at most 2, and $\alpha$ has a fixed point as well which we denote by $\bar w$ which means there is a periodic orbit $w$ in $M$ homotopic to the core of the torus.
Next, as the flow has orientable foliations and in particular the foliations of $w$ are orientable, $z^{\pm1}\sim w$.
Therefore, by Lemma~\ref{lem:isotopic to core}, the orbit $z$ is isotopic to the core of the torus $w$,
and thus the torus $T$ is boundary parallel in $\mg$.
\end{itemize}

\begin{figure}[t]
    \centering
    \begin{overpic}
    [width=0.4\textwidth]{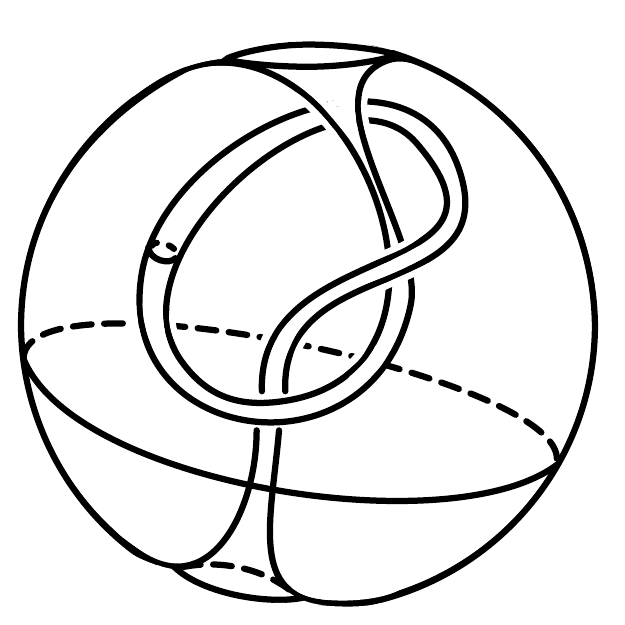}
    \end{overpic}
    \begin{overpic}
    [width=0.4\textwidth]{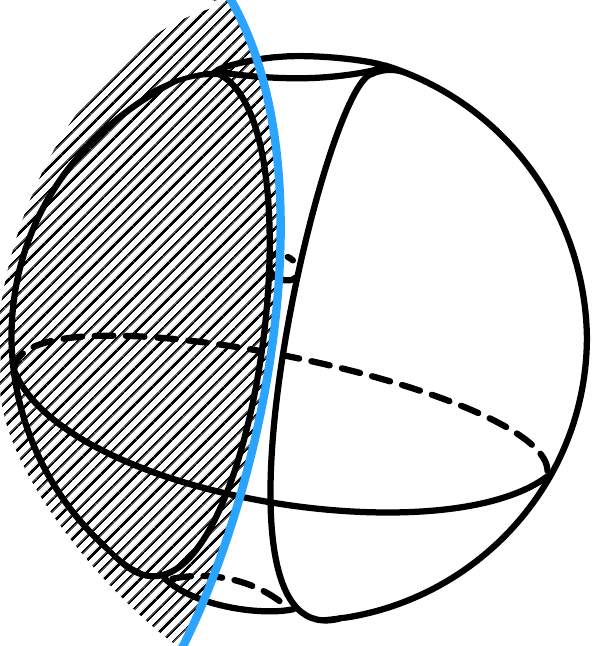}
    \put(40,97){$\gamma$}
    \end{overpic}
    \caption{A ball with a knotted tunnel on the left and with a trivial tunnel on the right.}
    \label{fig:tunnel}
\end{figure}
\end{proof}

\begin{lemma}\label{cor: toroidal implies 2,1 annuulus}
    Let $\gamma$ be a filling periodic orbit for an Anosov flow on a 3-manifold $M$. If $\mg$ is toroidal then $\gamma$ is a $(2,1)$ knot on the boundary of a solid torus. In particular, it is also annular.
\end{lemma}

\begin{proof} 
The only possibility of toroidal occurs in 
option (b) of the previous lemma, where $\gamma$ is
contained in a solid torus, and $\gamma = \alpha^2$,
where $\alpha$ represents a core of the solid torus
(we reverse orientation on $\alpha$ if $n = -2$). 
This implies that with the appropriate choice of 
longitude in this knot, it follows that $\gamma$ is
a $(2,1)$ knot in this torus, as by considering the cover of $M$ corresponding to $\gamma^2$ we get that $\gamma$ is isotopic to $\alpha^2$ by \ref{lem:isotopic to core}.

Finally, note that if $\gamma$ is a $(2,1)$ knot it clearly bounds a M\"obius band, and as the boundary of a neighborhood of the M\"obius band is an annulus with two boundary components on $\gamma$, $\gamma$ is annular in this case.
\end{proof}

\begin{corollary}\label{cor: nonorientable is atoroidal}
    Let $\gamma$ be a filling periodic orbit for an Anosov flow on a 3-manifold $M$. If $\gamma$ is nondivisible, or if $\gamma$ has nonorientable foliations, its complement  $\mg$ is atoroidal.
\end{corollary}

\begin{lemma}\label{lem:anannular}
   Let $\gamma$ be a filling periodic orbit of an Anosov flow $\phi$ with orientable foliations on a 3-manifold $M$. Then the complement $\mg$ is anannular, i.e. every annulus is compressible or boundary parallel.
\end{lemma}

\begin{proof}

Assume for the sake of contradiction that $\mg$ contains an essential annulus $A$. The annulus has two parallel boundaries on a boundary torus for a neighborhood of the orbit $\gamma$ that was removed from $M$. 

Neither of the boundary curves cannot be trivial in $M$, as no power of the periodic orbit $\gamma$ can bound a disk, and they cannot be meridional as $\gamma$ is prime and thus a meridional annulus is necessarily compressible. 

Thus, the boundary curves are two parallel nontrivial and non- meridional curves, that divide the boundary torus into two annuli $B$ and $B'$. 
Consider the unions $A\cup B$ and $A\cup B'$, slightly pushed to the inside of $\mg$. These are either both tori or both Klein bottles. 

First suppose they are tori. Since $\gamma$ is filling and does not intersect $A\cup B$ and $A\cup B'$, they must be compressible in $M$.
Consider the compressing disk for $A\cup B$. If $\gamma$ intersects it then, by the same argument as in the proof of Lemma \ref{lem:atoroidal}, either it bounds a solid torus with $\gamma$ its core, or, as $\gamma$ cannot pass essentially through a knotted tunnel unless the tunnel is trivial, there is another compression disk on the other side of $A\cup B$ as well, and it bounds a solid torus that does not contain $\gamma$. 

If $\gamma$ is a core of a solid torus bounded by $A\cup B$, the meridian of a neighborhood of $\gamma$ bounded by $B\cup B'$ is contained in the meridian of  $A\cup B$, thus the annulus $A$ is boundary parallel and we are done. Likewise for  $A\cup B'$. Therefore, we assume this is not the case.

Thus, $A\cup B$ bounds a solid torus that winds as a $(p,q)$ curve around the neighborhood of $\gamma$ (with $|p|>1$). The union of these two tori is a subset $N\subset M$ bounded by $A\cup B'$. By our assumption that $N$ is not a solid torus, the compression disk for $A\cup B'$ is not contained in $N$. As $\gamma\in N$ and $\gamma$ is not contained in a ball,  $A\cup B'$ bounds a solid torus that does not intersect $N$, and thus $M$ is a union of three solid tori, with $\gamma$ as the core of one of them.
If the meridian of the solid torus bounded by $A\cup B'$ has boundary parallel to the $(p,q)$ curves on the neighborhood of $\gamma$, the annulus $A$ is compressible in $M$. Assuming this is not the case, the Seifert fibration by these $(p,q)$ curves of $N$ extends into this solid torus.
Therefore, $M$ is a Seifert fiber space ($M$ has an orbit space of genus 0, and a maximal number of three singular fibers, so to carry an Anosov flow $M$ must be a small Seifert fiber space). However, an Anosov flow on a Seifert fiber space is orbit equivalent to a geodesic flow, up to covers, and carries no orbit that is homotopic to a fiber \cite{Barbot1996Waldhousen} (see also \cite{HammerlindlPotrieShannon} on the small Seifert fiber case). This rules out this possibility.

Finally, suppose that $A\cup B$ and $A\cup B'$ form Klein bottles. Consider a neighborhood of this Klein bottle, that is bounded by a torus $T=A_1\cup A_2\cup B\cup B'$, a double cover of the Klein bottle, where $A_1$ and $A_2$ are the two boundaries of a small neighborhood of $A$. $T$ bounds a twisted $I$-bundle over the Klein bottle on one side, that contains the orbit $\gamma$, and as $\gamma$ is filling $T$ has a compression disk. The compression disk must be on the other side of $T$, and as $\gamma$ is not contained in a three ball, $T$ bounds a solid torus on its other side.
However, in this case the manifold $M$ is the union of a twisted $I$-bundle over the Klein bottle and a solid torus, so $M$  cannot carry an Anosov flow by \cite{PlanteThurstonFundamentalGroup}, as its fundamental group does not have exponential growth.
\end{proof}

\begin{proof}[Proof of Theorems~\ref{thm: orbits are hyperbolic for orientable foliations}-\ref{thm: orbits are hyperbolic when not divisible}]

By Thurston \cite{ThurstonThreeDymensionalManifoldsKleinian}, since the boundary of $\mg$ is made up of tori, it is hyperbolic if and only if it does not contain an essential disk, sphere, torus or annulus.

We show that the manifold $\mg$ contains no essential  disk: suppose it did, and let $D$ be such a disk, with boundary $\alpha$. Then $\alpha$ is a
curve in the boundary of the solid torus neighborhood of $\gamma$.
We say that $D$ is meridional if $\alpha$ bounds a disk in $V$,
and we say that $D$ is longitudinal otherwise. If $D$ is meridional
then $D$ together with a disk in $V$ is a sphere $S$ in $M$ and since
$M$ is irreducible, then this sphere would bound a ball $B$ in $M$.
But this contradicts that $\gamma$ intersects $S$ only once,
showing that $S$ is non separating. Otherwise assume that
$D$ is longitudinal. Then $\alpha$ is homotopic in $M$ to a power
of $\gamma$. This implies that a power of $\gamma$ is null homotopic in $M$ and that is a contradiction.

It follows that $\mg$ contains no essential sphere by \ref{lem:irreducible}, no essential torus by \ref{lem:atoroidal}, and no essential annulus by \ref{lem:anannular}. Therefore $M_\gamma$ is 
hyperbolic.
\end{proof}

\begin{proof}[Proof of Corollary \ref{cor:main}]
It follows from Corollary \ref{cor: toroidal implies 2,1 annuulus} that an orbit $\gamma$ may be have a toroidal complement only when $\gamma$ is contained
in a solid torus and $\gamma = \delta^2$ in $\pi_1(M)$,
where $\delta$ represents a core of the solid torus.  This implies that $\delta$ reverses the orientations of the foliations of $\bar\delta$ when acting on the orbit space (equivalently it has nonorientable two dimensional foliations), while it preserves the orientations of the foliations through $\bar\gamma$, thus $\gamma$ itself has orientable foliations. This proves Corollary \ref{cor:main}.   
\end{proof}

\begin{corollary}\label{cor:nonorientable orbit is hyperbolic}
    Let $\gamma$ be a filling periodic orbit for an Anosov flow on a 3-manifold $M$. If $\gamma$ is nondivisible, or if $\gamma$ has nonorientable foliations, its complement  $\mg$ is hyperbolic.
\end{corollary}

%%%%%%%%%%%%%%%%%%%%%%%%%%%%%%%%%%%%%%%%%%%%%%%%%%%%%
\section{The complement of several orbits}
\label{sec:several orbits}

Next we turn to the question of taking the complement of several orbits of an Anosov flow. Since there exist Anosov flows with isotopic orbits \cite{BarthelemeFenley2014KnotTheory}, we cannot expect the result to always be hyperbolic. However, in algebraic flows the result holds. We first define the analog of Definition\ref{def:filling orbit} for a collection of several curves.

\begin{definition}\label{def:filling several orbit}
    Let $\gamma_1,\dots,\gamma_n$ be a collection of simple closed curves in an irreducible 3 dimensional manifold $M$. We say the collection is \emph{filling} if for any $K$ essential torus or Klein bottle in $M$, there exist an $1\leq i\leq n$ so that $\gamma_i$ intersects $K$. We say that the collection is \emph{anannular} if there is no properly embedded annulus in $M_{\gamma_1,\dots,\gamma_n}$.
\end{definition}

\begin{theorem}\label{pro: filling anannular set is hyperbolic}
    Let $\{\gamma_1\dots,\gamma_n\}$ be a filling and anannular collection of orbits in an Anosov flow. Then $M_{\gamma_1,\dots,\gamma_n}$ is hyperbolic.
\end{theorem}

\begin{proof}
 Lemma \ref{lem:irreducible} still proves that $M_{\gamma_1\dots,\gamma_n}$ is irreducible by lifting the sphere to the universal cover $\tilde M$.

Suppose $T$ is an essential torus in $M_{\gamma_1\dots,\gamma_n}$. $T$ has to have a compression disk in $M$, by the assumption that the collection $\{\gamma_1\dots,\gamma_n\}$ is filling. Thus, at least one of the orbits must puncture the compression disk in order for T to be essential in their complement.

Again there are two cases:
Either (1) $T$ bounds a solid torus $V$ and  at least one periodic orbit is contained in $V$. If there is a unique orbit $\gamma_i$ in the solid torus, then as it is annanular, it follows from Corollary \ref{cor: toroidal implies 2,1 annuulus} that it is homotopic to the core of the torus. Thus it is isotopic to the core by \Cref{lem:isotopic to core} of $V$, thus $T$ is boundary parallel in $M_{\gamma_1\dots,\gamma_n}$. If there are at least two orbits in $V$ all are equal to the core to some power and are jointly isotopic either to the core or to $(2,1)$ curves (cf. Lemma \ref{lem:atoroidal}). 
This implies there is an annulus bounded by two of the orbits in $V$, in contradiction to the assumption. 

The other option (2) is that $T$ is a 3-ball with a knotted tunnel, and one or more arcs of the orbits pass essentially through the tunnel. In this case the tunnel must be trivial and there is another compression disk to the other side as in the proof of \Cref{lem:atoroidal}. Thus either $T$ is compressible in $M_{\gamma_1\dots,\gamma_n}$ (if no orbit punctures the second compression disk), or there is at least one orbit puncturing the additional compression disk, in which case as in case (1), either the orbits are annular (contradicting the assumption) or $T$ is boundary parallel.

By assumption $M_{\gamma_1\dots,\gamma_n}$ contains no annulus and by Heafliger it contains no disk. This finishes the proof.
\end{proof}

The following example shows a collection of filling geodesics that is not anannular, and whose complement is not hyperbolic.  

\begin{example}
Consider the filling set of closed geodesics $\gamma_1,\gamma_2,\gamma_3,\gamma_4$ given in Figure~\ref{fig:genus2}. The two orbits $\gamma_2$ and $\gamma_3$ correspond to the same simple closed geodesic $\bar\gamma$ endowed with two different orientations. Thus, they co-bound two essential annuli in the torus $T$ that is the unit tangent bundle to this closed geodesic.
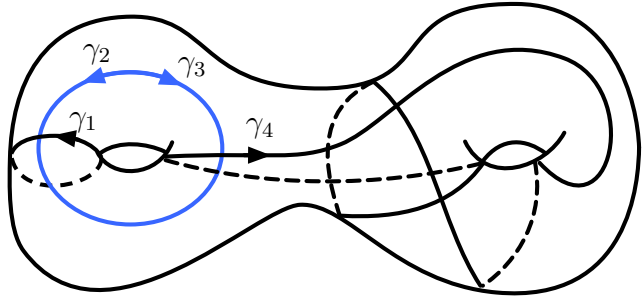
\begin{figure}[t]
    \centering
    \begin{overpic}[width=9 cm]{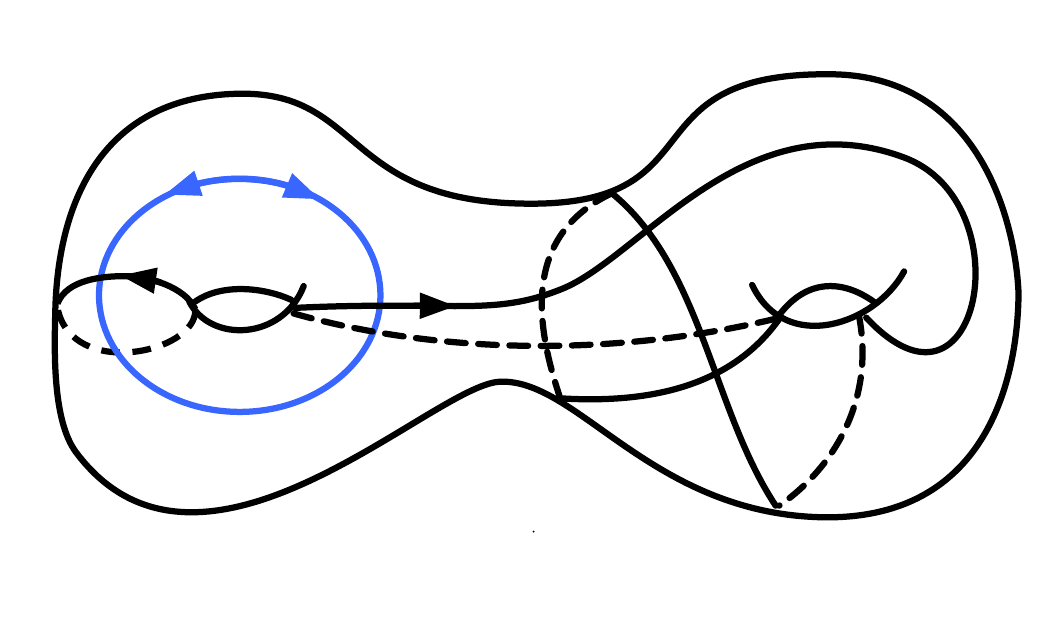}
        \put(14,35){$\gamma_1$}
        \put(16,45){$\gamma_2$}
        \put(30,43){$\gamma_3$}
        \put(40,34){$\gamma_4$}
    \end{overpic}
    \caption{A filling set of four geodesics on a genus two surface.}
    \label{fig:genus2}
\end{figure}
The torus $T$ is essential, and $\gamma_1$ and $\gamma_4$ both intersect it. However, rotating at each point of $\bar\gamma$ the unit vector belonging to $\gamma_2$  counterclockwise until reaching the direction of $\gamma_3$ defines one of the annuli composing $T$, and this annulus does not intersect $\gamma_1$ nor $\gamma_4$.
Thus, $M_{\gamma_1,\gamma_2,\gamma_3,\gamma_4}$ is annular and in particular is not hyperbolic.
\end{example}

%\begin{remark}
    Observe that for a suspension Anosov flow the only essential torus in the suspension manifold is, up to isotopy, a global cross-section. Thus, any collection of closed orbits is filling. For a geodesic Anosov flow, any collection of closed orbits that projects onto a filling collection of geodesics in the surface is filling. A collection of periodic orbits of the geodesic flow is called \emph{complete}, if each time the collection contains both orientations of the same simple closed geodesic, it also contains geodesics intersecting this simple closed geodesic in both directions.   
%\end{remark} 

\begin{proposition}
\label{prop:several_orbits_algebraic_flow}
Let $\{\gamma_1\dots,\gamma_n\}$ be either an arbitrary collection of periodic orbits in a suspension Anosov flow, or a filling and complete collection of periodic orbits of a geodesic flow. Then, the complement $M_{\gamma_1,\dots,\gamma_n}$ is hyperbolic.
\end{proposition}

\begin{proof}
First, let $\gamma$ be a single filling primitive periodic orbit in an algebraic Anosov flow. If it is not a power 2 orbit it has a hyperbolic complement by Theorem \ref{thm: orbits are hyperbolic for orientable foliations}, thus we may assume it is a power 2. Furthermore, by Lemma \ref{lem:power<=2}, there exists a periodic orbit $\alpha$ such that  $\gamma\sim\alpha^2$.
Both for a geodesic flow and in a suspension flow there exists at most one oriented orbit in any free homotopy class (For suspensions it follows from the work of Nielsen). Therefore, $\gamma$ is equal to $\alpha^2$, in contradiction to the assumption that $\gamma$ is primitive. This yields that for a single filling orbit $\mg$ is always hyperbolic.

Next, 
$M_{\gamma_1,\dots,\gamma_n}$ contains no essential annulus with both boundary components on a single orbit, as since the orbit is not a power 2 such an annulus defines an adjacent essential torus that no orbit of the collection $\{\gamma_1\dots,\gamma_n\}$ intersect it, in contradiction to the assumption that the collection is filling. 

$M_{\gamma_1,\dots,\gamma_n}$ contains no annulus that has boundary components on two different orbits, as this implies there exists an isotopy between one orbit $\gamma_i$ and either $\gamma_j$ or $\gamma_j^{-1}$. However this 
is impossible for algebraic Anosov flows only is $\gamma_i$ and $\gamma_j$ are the two orientation of the same simple closed geodesic. However, by our assumption, in this case both unique annuli between $\gamma_i$ and $\gamma_j$ are punctured by other orbits in the collection $\{\gamma_1,\dots,\gamma_n\}$.

Hence, the proof follows from \Cref{pro: filling anannular set is hyperbolic}.
\end{proof}

%%%%%%%%%%%%%%%%%%%%%%%%%%%%%%%%%%%%%%%%%%%%%%%%%

%%%%%%%%%%%%%%%%%%%%%%%%%%%%%%%%%%%%%%%%%%%%%%%%%

\section{An orbit with non hyperbolic complement}\label{sec:example}

It follows from the previous two sections that a filling orbit may have a non hyperbolic complement only if it divisible, which can happen only if the flow has nonorientable foliations and the orbit is a power 2 of a  `nonorientable' orbit. 
 In this section we give an example of such an orbit that has a non hyperbolic complement, demonstrating that  the condition of either the flow having orientable foliations or the orbit being non divisible is necessary.

\begin{proposition}
\label{prop:counter-example}
    There exists a transitive Anosov flow on a hyperbolic manifold, for which there exists a filling periodic orbit with a non hyperbolic complement.
\end{proposition}

\begin{proof}
Let $S_2$ be a genus 2 surface, and let $\alpha$ be a filling closed geodesic on $S_2$. Let $D_1$ and $D_2$ be two small disks in the complement of $\alpha$, and define a new nonorientable hyperbolic surface $S_c$ by gluing two M\"obius bands $M_1$ and $M_2$ so that $\partial M_i$ glues to $\partial D_i$. Denote by $\delta_i$ the core of the M\"obius band $M_i$.
Note that the unit tangent bundle to each $\delta_i$ is an essential Klein bottle $K_i$.

Adjusting via an isotopy in $S_2\setminus(D_1\cup D_2)$ and abusing notation, we assume 
$\alpha$ is a geodesic for the hyperbolic metric on $S_c$. Likewise we assume $\delta_1$ and $\delta_2$ are closed geodesics. Define $\beta$ to be a closed geodesic that does not intersect $\delta_1$, and intersects $\delta_2$ once. Finally, let $\gamma$ be an infinite geodesic, with $\alpha$ limit set equal to $\delta_1$, $\omega$ limit set equal to $\delta_2$, and so that it never intersects $\delta_1\cup\delta_2$.

For each $i=1,2$, denote by $\delta_i^{+}$ and $\tilde\delta_i^{-}$ the orbits of the geodesic flow that are lifts of $\delta_i$ to the unit tangent bundle in two opposite orientations. Observe that for each $i$, The two orbits are contained in a Klein bottle that is the unit tangent bundle to $\delta_i$.

\begin{figure}[t]
    \centering
    \begin{overpic}[width=12 cm]{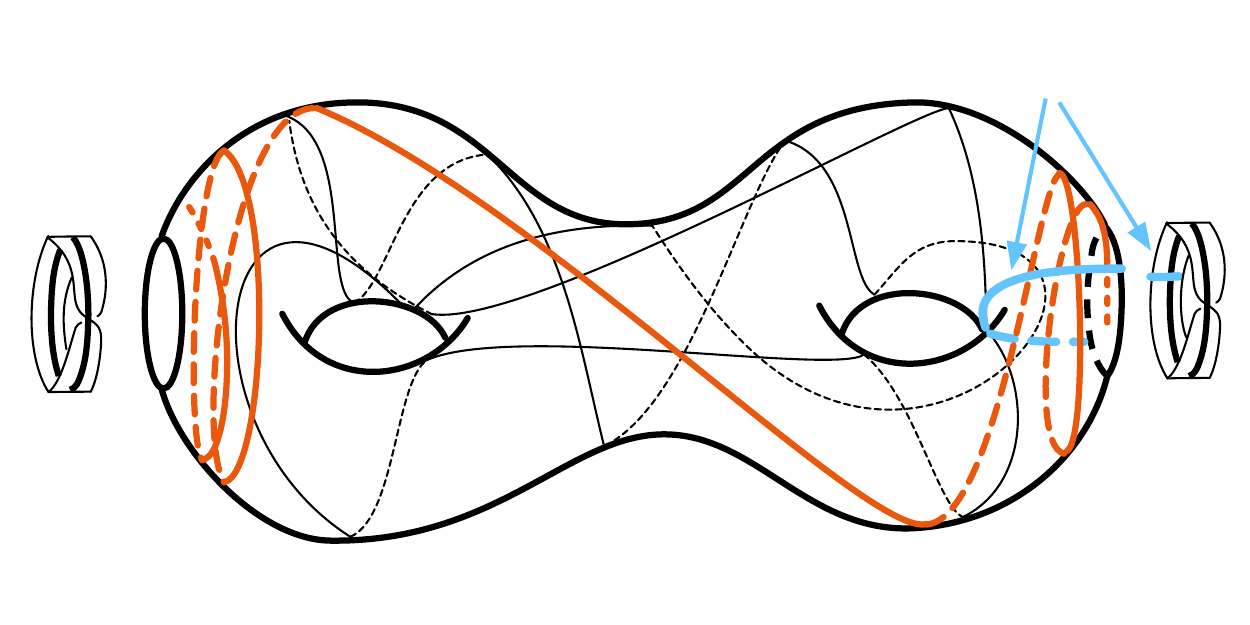}
    \put(68,12){$\gamma$}
    \put(32,12){$\alpha$}
    \put(5,15){$\delta_1$}
    \put(95,16){$\delta_2$}
    \put(83,43){$\beta$}
    \end{overpic}
    \caption{The genus 2 surface $S_2$ to which we gule two M\"obius bands, with a few geodesics that are used in the proof.}
    \label{fig:example}
\end{figure}
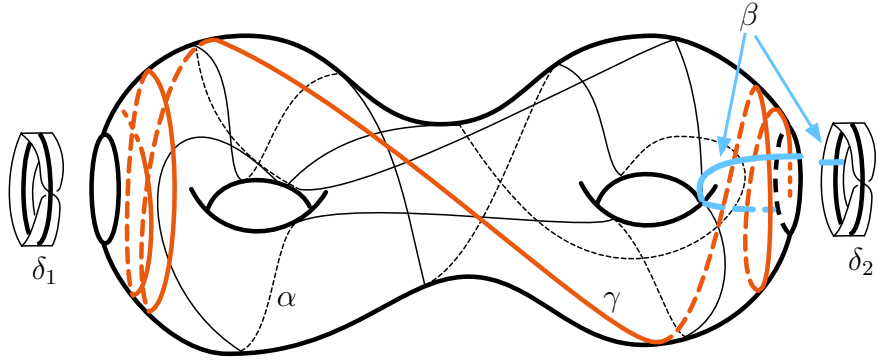

Make an expanding DA along $\delta_1^{-}$ together with a contracting DA along $\delta_2^{-}$. By removing suitable tubular neighborhoods of both orbits, we obtain a 3-manifold with two torus boundaries: $T_1$ where the flow is entering, and $T_2$ where the flow is escaping.

For each $i=1,2$, after making the DA operation there is an extra orbit $\epsilon_i$ that appears near the perturbed region, has orientable weak stable and weak unstable manifolds, and is isotopic to $(\delta_i^{+})^2$ through an embedded M\"obius band $B_i$. This Mobius band consists of a Mobius band $A_i^{+}\subset T^1(\delta_i)$, an annulus inside $T_i$, and a half stable (resp. unstable) manifold of $\epsilon_i$.

When making the DA operations, the original orbits of the geodesic flow on $S_c$ will ``survive" the construction. That is, if we call $\varphi_t$ the geodesic flow on $T^1 S_c$ and $\tilde \varphi$ the flow obtained after the DA operation, then there is a $\tilde\varphi$-invariant compact set $\Lambda\subset T^1 S_c$ and a surjective continuous map $h:\Lambda\to T^1 S_c$ taking $\tilde\varphi$-orbits onto $\varphi$-orbits of the geodesic flow (cf. \cite{beguin2017building}). In particular, there is an orbit segment $\tilde\gamma^{+}$ of $\tilde\varphi$ that such that $h:\tilde\gamma^{+}\mapsto\gamma^{+}$, and this orbit segment verifies that it starts at $T_1$, ends at $T_2$, and is disjoint from the M\"obius band $B_1$. Without loss of generality, we may assume that $\tilde\gamma^{+}$ equals the geodesic segment $\gamma^{+}$. 

The resulting manifold with boundary carries a vector field with a hyperbolic $\tilde\varphi$-invariant set inside, and the stable (resp. unstable) foliations of this set hit $T_1$ (resp. $T_2$) in a Reeb-like fashion. Thus, gluing $T_1$ and $T_2$ with a diffeomorphism that makes these foliations transverse along the boundary results in a new manifold $N$ endowed with an Anosov flow $\psi$, that is also transitive. We refer to \cite{beguin2017building} for more details on gluing hyperbolic pieces along the boundary to obtain Anosov flows.

We will choose the gluing diffeomorphism in such a way that the endpoint of $\gamma^{+}$ in $T_2$ is glued to its starting point in $T_1$. In this way, we produce a closed orbit for $\psi$ that we denote by $\gamma$, that hits transversally the gluing torus $T_1\sim T_2$, that we now denote simply by $T$.

The union of orbits $\alpha\cup\beta\cup\gamma$ is filling, intersecting every essential torus in the Seifert fibered piece $T^1(S_2)$, intersecting $T$, and intersecting (thanks to $\beta$) the Klein bottle obtained as the union of the two annuli that are the complements of $\delta_1^{-}$ and $\delta_2^{-}$ in $T^1_{\delta_1} S_c$ and $T^1_{\delta_1} S_c$, that are attached to each other after gluing $T_1$ to $T_2$. It also holds that $N_{\alpha,\beta,\gamma}$ is anannular. This follows because in the Seifert fiber piece every annulus is be vertical. Hence, $N_{\alpha,\beta,\gamma}$ is hyperbolic by \Cref{pro: filling anannular set is hyperbolic}.
Now, if we perform Fried surgeries along these orbits with high enough coefficients, we obtain a hyperbolic 3-manifold $M$ carrying an Anosov flow $\phi$.

Finally, consider the orbit $\varepsilon_1$ of $\phi$. As $M$ is hyperbolic it is atoroidal and thus $\varepsilon_1$ is trivially filling. However, $\varepsilon_1$ is the boundary of a M\"obius band in $M$ and the boundary of a small neighborhood of this band is an essential torus in $M_{\varepsilon_1}$. Thus $\varepsilon_1$ is a filling orbit with a non hyperbolic complement, as required.
\end{proof}

\bibliographystyle{amsplain}
\bibliography{bibliography.bib}

\end{document}